\documentclass[12pt,reqno,a4paper]{amsart}
\usepackage[USenglish]{babel}
\usepackage[utf8]{inputenc}
\usepackage{amsfonts}
\usepackage{amssymb}
\usepackage{amsthm}
\usepackage{todonotes}
\usepackage{enumitem}
\usepackage{comment}
\usepackage{tikz}
\usepackage{ upgreek }
\usepackage{dsfont}
\usepackage{epic}
\usepackage{indentfirst}
\usepackage[a4paper,top=3.3cm,bottom=3cm,left=3cm,right=3cm,bindingoffset=5mm]{geometry}
\usepackage[pdftitle={},
pdfauthor={Matthias Nickel},
linktocpage=true
]{hyperref}
\usepackage{breqn}

\DeclareMathOperator{\vol}{vol}

\DeclareMathOperator{\ord}{ord}
\DeclareMathOperator{\ind}{ind}
\DeclareMathOperator{\Ker}{Ker}

\newtheorem{theorem}{Theorem}
\newtheorem*{theorem*}{Theorem}
\newtheorem{lemma}[theorem]{Lemma}
\newtheorem{prop}[theorem]{Proposition}
\newtheorem{cor}[theorem]{Corollary}

\newtheorem{defin}[theorem]{Definition}
\theoremstyle{remark}

\newcommand{\BA}{{\mathbb{A}}}
\newcommand{\BC}{{\mathbb{C}}}
\newcommand{\BN}{{\mathbb{N}}}
\newcommand{\BP}{{\mathbb{P}}}
\newcommand{\BQ}{{\mathbb{Q}}}
\newcommand{\BR}{{\mathbb{R}}}

\newcommand{\BZ}{{\mathbb{Z}}}

\newcommand{\mcI}{{\mathcal{I}}}
\newcommand{\mcJ}{{\mathcal{I}}}
\newcommand{\mcO}{{\mathcal{O}}}

\newcommand{\eps}{{\varepsilon}}

\begin{document}

\setlength{\parindent}{2ex}

\title[Local positivity and effective Diophantine approximation]{Local positivity and effective Diophantine approximation}

\begin{abstract}
In this paper we present a new approach to prove effective results in Diophantine approximation. We then use it to prove an effective theorem on the simultaneous approximation of two algebraic numbers satisfying an algebraic equation with complex coefficients. 
\end{abstract}

\author[M. Nickel]{Matthias Nickel}
\address{Johannes Gutenberg-Universit\"at, Institut f\"ur Mathematik, Staudingerweg 9, D-55128 Mainz, Germany}
\email{manickel@uni-mainz.de}

\date{\today}

\maketitle

\section{Introduction}
Positivity concepts for divisors play a crucial role in algebraic geometry. Among these concepts is \emph{ampleness}, which can also be interpreted intersection theoretically via the Nakai--Moishezon--Kleiman criterion. A weaker form of positivity is \emph{bigness}: a divisor $D$ is big iff the growth of the dimension of global sections of its multiples is maximal. The rate of this growth is then measured by the volume of the divisor \cite[Section 2.1]{PAG} and for ample divisors this is simply the top self-intersection by the asymptotic Riemann--Roch theorem \cite[Theorem 1.1.24]{PAG}. In \cite{Demailly} Demailly introduces a measure of local positivity of a divisor at a point, the Seshadri constant, in order to study the Fujita conjecture.\\

The connection between Diophantine approximation and positivity concepts is central to many results on Diophantine geometry. It is a key element in Vojta's proof of Mordell's conjecture \cite{VojtaSiegel} and in Faltings's proof of the Mordell--Lang conjecture \cite{F}. In \cite{FW} it has been shown that the constants showing up in Diophantine approximations can be obtained as the expectation of certain random variables coming from filtrations on the graded ring of sections of a divisor. Later \cite{FerrettiMumford,EF,CorvajaZannier,EF2} showed that these constants can be obtained via different geometric invariants. In \cite{mcroth} Diophantine approximation constants are shown to be related to volumes of divisors. In \cite{RuVojta}, \cite{RuWang} and \cite{HeierLevin} this result is extended to the more general case where not only points but closed subschemes are approximated.\\

Most results on Diophantine approximation rely on the construction of an auxiliary polynomial having a certain order of vanishing at given points. In this paper we present a new approach that follows Faltings's proof of the Mordell--Lang conjecture \cite{F} using information on local positivity at these points to study the vector spaces of suitable auxiliary polynomials.\\

One of the most important results in Diophantine approximation is Roth's theorem on the approximation of algebraic numbers by rationals \cite{KFRoth}. It states that for a given algebraic number $\alpha$ and a given $\eps > 0$ there are only finitely many rational numbers $p/q \in \BQ$ such that 
\begin{equation}
\label{eq:Roth}
\left|  \alpha - \frac{p}{q} \right| \leq q^{-(2+\eps)} \,.
\end{equation}

The proof of this theorem consists of two steps:

\begin{enumerate}[label=\arabic*.] 
\item First an auxiliary polynomial $P \in \BZ[X_1,\dots,X_n]$ having a certain order of vanishing at $(\alpha,\dots,\alpha)$ is constructed, which is then shown to vanish to a suitable order at $(p_1/q_1,\dots,p_n/q_n)$ where $p_i/q_i$ are solutions to \ref{eq:Roth}.
\item Next, one shows that there exists an upper bound for the order of $P$ at the point $(p_1/q_1,\dots,p_n/q_n)$ obtaining a contradiction. This upper bound may be either of geometric (Dyson's lemma \cite{Dyson} or rather its generalization by Esnault and Viehweg \cite{EV}) or of arithmetic nature (Roth's lemma \cite{KFRoth} and Faltings's product theorem \cite{F}).
\end{enumerate}

There are also many results on the simultaneous approximation of algebraic numbers by rationals. The generalization of Roth's theorem in this context is due to Schmidt \cite[Corollary to Theorem 1]{Schmidt}. Suppose that $\alpha_1,\dots,\alpha_r$ are algebraic numbers such that  $1,\alpha_1,\dots,\alpha_r$ are linearly independent over $\BQ$. Then for every $\varepsilon > 0$ there exist only finitely many $r$-tuples of rational numbers $(p_1/q,\dots,p_r/q)$ such that
\begin{equation}
\label{eq: Schmidt chapter}
\left|\alpha_i - \frac{p_i}{q} \right| \leq q^{-(1+1/r+\varepsilon)} 
\end{equation}
holds for all $1\leq i\leq r$.\\ 

Note that the theorems of Roth and Schmidt are not effective in that there is no bound for $q$ satisfying \ref{eq:Roth} and \ref{eq: Schmidt chapter} respectively. The earliest effective result in the approximation of a single algebraic number is the theorem of Liouville \cite{Liouville}, which is similar to Roth's theorem with exponent the degree $d$ of the algebraic number in question instead of $2+\varepsilon$. Fel'dman \cite{Feldman} obtained an improvement of Liouville's theorem, in which the exponent is strictly smaller than $d$, however, the difference is extremely small. For improvements and a different approach see \cite{BuGy,BuBorne,BoGm}. In the case of simultaneous approximation there are effective results where the tuple of algebraic numbers is given by rational powers of rational numbers \cite{BakerSimul,Osgood,Rickert,Bennett}.\\

Here we discuss a different strategy linking methods from positivity and Diophantine approximation that follows Faltings's proof of the Mordell--Lang conjecture \cite{F}. We consider homogeneous polynomials in two variables having large index at the point $(\alpha_1,\alpha_2)$ and \emph{a priori} small index at $(p_1/q,p_2/q)$ where $p_i/q$ is a suitably good rational approximation of $\alpha_i$ for $i=1,2$. \\
Using Faltings's Siegel lemma we can then ensure that we can find such a polynomial with suitably bounded coefficients in $\BZ$. Finally we give a bound for $q$ involving the index of $P$ at $(\alpha_1,\alpha_2)$ and $(p_1/q,p_2/q)$.\\

The novelty of this approach is that it avoids providing a zero estimate: we only need to suitably bound the dimension of the space of polynomials with given degree and given index at $(\alpha_1,\alpha_2)$, all of its conjugates and $(p_1/q,p_2/q)$. Therefore we only need a partial understanding of the volume function on blowups of $\BP^2$. The fact that we only consider one solution $(p_1/q,p_2/q)$ will finally make our theorem effective.\\

We obtain the following theorem.

\begin{theorem}
\label{thm:main}
Let $\alpha_1, \alpha_2$ be algebraic numbers and let $d:=[\BQ(\alpha_1,\alpha_2):\BQ]$. Suppose that $(\alpha_1,\alpha_2)$ and all of its conjugates lie on an irreducible curve of degree $m$ defined over $\BC$. Then there exists for all $\delta \in \BQ$ with $\delta > \max\{m,d/m\}$ an effectively computable constant $C_0(\alpha_1,\alpha_2,\delta,m)$ depending only on $(\alpha_1,\alpha_2)$, $m$ and $\delta$ such that for all pairs of rational numbers $(p_1/q,p_2/q)$ satisfying
\begin{equation}
\label{eq:main}
\left| \alpha_i - \frac{p_i}{q} \right| \leq q^{-\delta} \text{ for i=1,2}
\end{equation} 
we have $q \leq C_0(\alpha_1,\alpha_2,\delta,m)$.
\end{theorem}

Note that the strength of the theorem lies in the fact that the irreducible curve may be defined over $\BC$ and not only $\BQ$.\\We state a suitable choice for $C_0(\alpha_1,\alpha_2,\delta,m)$ in Corollary \ref{cor:C0}.

\subsection{Acknowledgments}

I would like to thank Matteo Costantini, Jan-Hendrik Evertse, Walter Gubler, Ariyan Javanpeykar, Lars Kühne, Victor Lozovanu, Martin Lüdtke, Marco Maculan, David McKinnon, Mike Roth, Jakob Stix, Paul Vojta and Jürgen Wolfart for many helpful discussions and suggestions. Furthermore, I would like to express my gratitude towards my former thesis advisor Alex Küronya for his support and many useful comments. \\The author was partially supported by the LOEWE grant ``Uniformized Structures in Algebra and Geometry''.

\subsection{Notation} In the remainder of this article we will denote by $\alpha_1$ and $\alpha_2$ algebraic numbers and let $d:=[\BQ(\alpha_1,\alpha_2):\BQ)]$.


\section{Seshadri constants and Zariski-decomposition on Blow-ups of $\BP^2$}

In this section we will be only concerned with varieties over $\BC$.\\

We begin by discussing Seshadri constants. These constants measuring local positivity of divisors were first defined by Demailly in \cite{Demailly} and their name is due to the Seshadri criterion for ampleness \cite[Remark 7.1]{Seshadri}.

\begin{defin}
Let $X$ be a smooth projective surface, let $M$ be a nef divisor on $X$ and let $x, x_1, x_2,\dots,x_n$ be points in $X$. Then the Seshadri constant of $M$ at $x$ is defined as
\[ \eps(X,M;x):= \sup \{ t \geq 0 \mid \pi_x^* M - t E \text{ is nef on } X'\} \]
where $\pi_x : X' \rightarrow X$ is the blowup of $X$ at $x$ and $E$ its exceptional divisor. Furthermore the multi-point Seshadri constant of $M$ at $x_1,\dots,x_n$ is defined as 
\[ \eps(X,M;x_1,\dots,x_n):= \sup \{ t \geq 0 \mid \pi_{x_1,\dots,x_n}^* M - t (E_1+\dots+E_n) \text{ is nef on } X''\} \]
where $\pi_{x_1,\dots,x_n}: X'' \rightarrow X$ is the blowup of $X$ at the points $x_1,\dots,x_n$ with exceptional divisors $E_1,\dots,E_n$.     
\end{defin}

The following lemma is well known. We provide a proof for the convenience of the reader.

\begin{lemma}
\label{ample}
Let $X''$ and $x_1,\dots,x_n$ be as above and let $M$ be ample. Then we have that $\pi_{x_1,\dots,x_n}^* M - t (E_1 +\dots+E_n)$ is ample for $t < \eps(X,M;x_1,\dots,x_n)$. 
\end{lemma}

\begin{proof}
If $\pi_{x_1,\dots,x_n}^* M - t (E_1+\dots+E_n)$ is not ample for some $t < \eps(X,M;x_1,\dots,x_n)$, there exists a curve $C$ in $X''$ such that $(\pi_{x_1,\dots,x_n}^* M - t (E_1+\dots+E_n)) C = 0$. First note that $(\pi_{x_1,\dots,x_n}^* M - t (E_1+\dots+E_n)) E_i = t > 0$ for $i=1,\dots,n$, which shows that $C$ must be the strict transform of a curve on $X$. Now if $E_1+\dots+E_n$ does not intersect $C$, we have that $C$ is the pullback of a curve on $X$ not containing $x_1,\dots,x_n$ and this is impossible as $M$ is ample. Hence $(E_1+\dots+E_n) C > 0$ and therefore $(\pi_{x_1,\dots,x_n}^* M - t_0 E) C < 0$ for $t < t_0 < \eps(X,M;x_1,\dots,x_n)$ yields a contradiction. 
\end{proof}

Let us recall some properties of Seshadri constants.

\begin{lemma}[{{\cite[Example 5.1.4, Example~5.1.6]{PAG}}}]
\label{properties}
Let $X,X'$ and $x$ be as above and let $M$ be nef. Then:
\begin{enumerate}[label=\arabic*.]
	\item The Seshadri constant is homogenous:
	\[\eps(X,l M;x) = l \, \eps(X,M;x)\]
	for all $l \in \BN$.
	\item If $M$ is very ample then \[\eps(X,M;x_1,\dots,x_n) \geq 1 \,.\]
\end{enumerate}
\end{lemma}

For more about Seshadri constants the reader may consult \cite[Chapter 5]{PAG} and \cite{primer}.

We will use the following lemma, which is a direct consequence of \cite[Lemma 14.15]{Badescu} on the minimality of the negative part in Zariski decompositions. This argument was also used in \cite[Proposition 2.1]{KLM12}.

\begin{lemma}
\label{minimality}
Let $D$ be a pseudoeffective divisor on a smooth projective surface $X$ and let $D=P+N$ be its Zariski decomposition. For any decomposition $D=P'+ N'$ with $P'$ nef and $N'$ effective we have $N \leq N'$. 
\end{lemma}

Let us now compute the multi-point Seshadri constant of $L=\mcO_{\BP^2}(1)$ at points that lie on an irreducible curve of degree $m$.

\begin{prop}
\label{lowerSeshadri}
Let $x_1,\dots,x_{d}$ be distinct points lying on a curve of degree $m$ in $\BP^2$, let $L$ be a line in $\BP^2$ such that $\mcO_{\BP^2}(L)=\mcO_{\BP^2}(1)$ and consider the blow-up $S$ of $\BP^2$ at $x_1,\dots,x_{d}$ with exceptional divisors $E_1,\dots,E_{d}$. Then for every $0 \leq t < \min \{1/m,m/d\}$ the divisor $L - t (E_1 + \dots + E_{d})$ is ample.
\end{prop}  

\begin{proof}
As $x_1,\dots,x_d$ lie on a curve of class $m$ we obtain that the class \[C:=m L-(E_1+\dots+E_{d})\] is effective on $\text{Bl}_{x_1,\dots,x_{d}} \BP^2$. This implies that for $0 \leq t \leq 1/m$ we have \[ L_{t} := L - t (E_1 + \dots + E_{d}) = (1-m t) L + t C \] is a decomposition as a sum of a nef and an effective divisor and Lemma \ref{minimality} shows that the support of the negative part of $L - t (E_1 + \dots + E_{d})$ must either be empty or may only contain $C$ and $E_i$ for $i=1,\dots,d$. Using 
\begin{align*}
L_t E_i &= t\\
L_{t} C &= m - d t\\
L_{t}^2 &= 1 - d t^2
\end{align*}
we conclude that $L_{t}$ is nef for $0 \leq t \leq \min \{1/m,m/d,1/\sqrt{d}\}$. Furthermore, as $1/m \leq m/d$ is equivalent to $1/m \leq 1/\sqrt{d}$, we conclude that the equality $\min \{1/m,m/d,1/\sqrt{d}\} = \min \{1/m,m/d\}$ holds and together with Lemma \ref{ample} this finishes the proof. 
\end{proof}

In what follows we will need to have a lower bound for the Seshadri constant of $L - t (E_1 + \dots + E_{d})$ for $0 \leq t < \min \{1/m,m/d\}$ at another point. In order to do this we will employ an effective version of Matsusaka's big theorem \cite{Matsusaka2,MM} for surfaces by Fern\'{a}ndez del Busto \cite{Fernandez}. Note that Siu has given an effective version of Matsusaka's big theorem valid in higher dimensions \cite{Siu,Siu2}.

\begin{theorem}[{{\cite{Fernandez}}}]
\label{Fernandez}
Let $A$ be an ample divisor on a smooth projective algebraic surface $X$. Then $l A$ is very ample for every 
\[ l > \frac{1}{2} \left\lfloor \frac{(A(K_X + 4 A)+1)^2}{A^2} + 3 \right\rfloor \,.\]
\end{theorem}

Using this we are now ready to prove the following geometric theorem, which will be essential in the proof of the main theorem.

\begin{theorem}
\label{voldiff}
Let $x_1,\dots,x_{d}$ be distinct points lying on an irreducible curve of degree $m$ in $\BP^2$, let $x_{d+1} \in \BP^2$, let $L$ be a line in $\BP^2$ and consider the blow-up $X$ of $\BP^2$ at $x_1,\dots,x_{d+1}$ with exceptional divisors $E_1,\dots,E_{d+1}$. Let us define for $Q>0$
\[l(\theta):=\frac{1}{2} \left\lceil \frac{((-3+d \theta)Q+4 Q^2 (1-d \theta^2)+1)^2}{Q^2(1-d \theta^2)} + 3\right\rceil .\]
Then for all $\theta \in \BQ$ with denominator bounded by $Q$ satisfying $\theta<\min \{1/m,m/d\}$ and for every $0 \leq \mu \leq \frac{1}{Q l(\theta)}$ we have that 
\[ \vol_{X} (L - \theta(E_1 + \dots + E_{d})) - \vol_{X} (L - \theta (E_1 + \dots + E_{d}) - \mu E_{d+1}) = \mu^2\,. \]
\end{theorem}

\begin{proof}
By Proposition \ref{lowerSeshadri} above we know that $M:=(L - \theta (E_1 + \dots + E_{d}))$ is ample. Note that $l(\theta) = \frac{1}{2} \left\lceil \frac{(Q M(K_X + 4 Q M)+1)^2}{(Q M)^2} + 3 \right\rceil$. By Theorem \ref{Fernandez} and because $Q M$ is an integral ample divisor, we now know that the divisor $l(\theta) Q M$ is very ample and therefore $\eps(\text{Bl}_{x_1,\dots,x_{d}}(\BP^2),M; x_{d+1}) \geq \frac{1}{Q l(\theta)}$ by the properties of Seshadri constants in Lemma \ref{properties}. Therefore $L - \theta(E_1 + \dots + E_{d})-\mu E_{d+1}$ is nef for $0 \leq \mu \leq \frac{1}{Q l(\theta)}$ and the statement of the theorem follows by the asymptotic Riemann--Roch theorem \cite[Corollary 1.4.41]{PAG}. 
\end{proof}


\section{Bounding the denominator of a good approximation}
In this section we will bound the denominator $q$ of a good approximation $(p_1/q,p_2/q)$ of $(\alpha_1,\alpha_2)$ using a polynomial $P \in \BZ[X_1,X_2]$ with suitably bounded coefficients and suitable index at $(\alpha_1,\alpha_2)$ and $(p_1/q,p_2/q)$. This chapter closely follows \cite[\S D.5]{HS}. For the convenience of the reader we give proofs as we need slightly different statements than those in \cite{HS}. \\

\begin{defin}
Let \[P = \sum_{j \in \BN_0^2} \, a_{j_1,j_2} {X_1}^{j_1} {X_2}^{j_2}\] be a polynomial with coefficients in $\BC$.
\begin{enumerate}[label=\arabic*.]
    \item For a multi-index $j \in \BN_0^2$ we define a differential operator $\partial_j$ via \[\partial_j P := \frac{1}{j_1! j_2!} \frac{\partial^{j_1+j_2}}{\partial X_1^{j_1} \partial X_2^{j_2}} P .\]
    \item The index of $P$ at $x=(x_1,x_2)$ with respect to the weights $(r_1,r_2) \in \BN^2$ is the nonnegative real number \[\ind_{(x_1,x_2;r_1,r_2)}(P) := \min \{ j_1/r_1+j_2/r_2 \mid j \in \BN_0^2, \partial_j P (x) \neq 0\}.\]
		\item If $P \in \BZ[X_1,X_2]$, the naive height of $P$ is defined as \[|P| := max \{ |a_j| \,\mid j \in \BN_0^2\} \,.\]
\end{enumerate}  
\end{defin}

Let us summarize some properties of the index and the differential operators $\partial_j$ for later use.

\begin{lemma}[{{\cite[Lemma D.3.1, Lemma D.3.2]{HS}}}]
\label{naivediff}
Let $P \in \BC[X_1,X_2]$ be a polynomial such that its degree is bounded by $k$. Then:
\begin{enumerate}[label=\arabic*.]
		\item $\ind_{(x_1,x_2;r_1,r_2)}(\partial_{j} P) \geq \ind_{(x_1,x_2;r_1,r_2)}(P) - j_1/r_1-j_2/r_2$,
    \item if $P \in \BZ[X_1,X_2]$, then $\partial_{j} P \in \BZ[X_1,X_2]$ for all $j \in \BN_0^2$,
    \item $|\partial_{j} P| \leq 4^{k} |P|$.
\end{enumerate}
\end{lemma}

From now on we will make the assumption $r_1 = r_2 = k$. In particular we have that $\ind_{(\alpha_1, \alpha_2; k, k)}(P) = \ord_{(\alpha_1, \alpha_2)}(P)/k$.\\
In the following two lemmas we will provide a bound on the denominator and absolute value of derivatives of a polynomial $P \in \BZ[X_1,X_2]$ at $(p_1/q,p_2/q)$.

\begin{lemma}
\label{heightdiff}
Let $P \in \BZ[X_1,X_2]$ be a polynomial of degree less or equal $k$ and let $j \in \BN_0^2$ be a multi-index. Assume that $\partial_j P(p_1/q,p_2/q) = s/m$ where $s \in \BZ, m \in \BN$ and $s,m$ are coprime. Then \[ m \leq q^{k}\,.\]
\end{lemma}

\begin{proof}
Using Lemma \ref{naivediff} we obtain that the the coefficients of $\partial_j P$ are in $\BZ$. Therefore $\partial_j P(p_1/q,p_2/q)$ is a sum of terms whose denominators are divisors of $q^{k}$ giving us the desired bound.
\end{proof}

\begin{lemma}
\label{normdiff}
Let $P \in \BZ[X_1,X_2]$ of degree less or equal $k$ with $k \geq 6$ and let $j \in \BN_0^2$ be a multi-index. Let $\theta$ be the index of $P$ at $(\alpha_1,\alpha_2)$ with respect to $(k,k)$, let $0<\theta_0<\theta$, let $\delta>0$ and let $N \in \BN$.\\
Then it holds that for $(p_1/q,p_2/q) \in \BQ$ satisfying 
\begin{equation}
\label{eq:deltaroth}
\left|\alpha_i - \frac{p_i}{q} \right| \leq N q^{-\delta} \text{ for $i=1,2$}
\end{equation}
and for every $j=(j_1,j_2) \in \BN_0^2$ such that $\frac{j_1+j_2}{k} \leq \theta_0$ we have \[ \left| \partial_j P(p_1/q,p_2/q) \right| \leq 64^k \left|P \right| (\max\{\left|\alpha_1\right|,\left|\alpha_2\right|\})^{k} N^{2 k} q^{-k \delta (\theta - \theta_0)}.\]     
\end{lemma}
 
\begin{proof}
First note that for all $i \in \BN_0^2$ we have that $\partial_i\partial_j P(\alpha_1,\alpha_2)$ is a sum of at most $1/2 \, (k+1) (k+2)  \leq 2^k$ terms because $k \geq 6$. These terms are themselves bounded by \[\left|\partial_i \partial_j P \right| (\max\{\left|\alpha_1\right|,\left|\alpha_2\right|\})^{k} \leq 16^k |P| (\max\{\left|\alpha_1\right|,\left|\alpha_2\right|\})^{k}\] where we have used Lemma \ref{naivediff} two times. We may now expand $\partial_j P$ around $(\alpha_1,\alpha_2)$ and use that Lemma \ref{naivediff} implies $\ind_{(\alpha_1,\alpha_2;k,k)}(\partial_j P) \geq \theta - \theta_0$ to obtain
\begin{equation*} \partial_j P (p_1/q,p_2/q) = \sum_{\substack{1 \leq i_1,i_2 \leq k \\(i_1+i_2)/k \geq \theta - \theta_0}} (\partial_i \partial_j P)(\alpha_1,\alpha_2) (p_1/q - \alpha_1)^{i_1} (p_2/q - \alpha_2)^{i_2}
\end{equation*}
and by assumption \ref{eq:deltaroth}, the fact that the number of terms above is bounded by $(k+1)^2 \leq 2^k$ and the bounds above we have
\[\left|\partial_j P (p_1/q,p_2/q) \right| \leq 64^k |P| (\max\{\left|\alpha_1\right|,\left|\alpha_2\right|\})^{k} N^{2 k}  q^{-k \delta (\theta - \theta_0)} \,.  \]
\end{proof}

Using the results above we obtain a bound for the denominator of a good approximation as follows.

\begin{lemma}
\label{lower}
Let $k \geq 6$ be a positive integer. Suppose that $(p_1/q,p_2/q) \in \BQ^2$ is a solution of inequality \ref{eq:deltaroth} for given $\delta > 1/(\theta - \theta_0)$, $N \in \BN$ and let $0 < \theta_0 < \theta$ be given. 
Now assume that $P \in \BZ[X_1,X_2]$ satisfies the following properties:
\begin{enumerate}[label=\arabic*.]
    \item the degree of $P$ is at most $k$,
    \item the index of $P$ at  $(\alpha_1,\alpha_2)$ with respect to the weights $(k,k)$ satisfies \[\ind_{(\alpha_1,\alpha_2;k,k)}(P) \geq \theta,\]
    \item $|P| \leq B^{k}$, where $B$ depends only on $(\alpha_1,\alpha_2)$, $k$ and $\delta$.
\end{enumerate}
Let \[C(\alpha_1,\alpha_2,\delta,N) :=\left(64 B  \max\{\left|\alpha_1\right|,\left|\alpha_2\right|\} N^{2} \right)^{\frac{1}{\delta (\theta - \theta_0)-1}} \,.\]
Then it holds that if \[ \ind_{(p_1/q, p_2/q; k, k)}(P) < \theta_0\,, \] we have $q \leq C(\alpha_1,\alpha_2,\delta,N)$.
\end{lemma}

\begin{proof}
Assume $\ind_{(p_1/q, p_2/q; k, k)}(P) < \theta_0$ and let $j \in \BN_0^2$ with $\frac{j_1+j_2}{k} < \theta_0$ be such that $\partial_j P(p_1/q, p_2/q) \neq 0$. Now Lemma \ref{normdiff} and the bound on $|P|$ give us that
\[
|\partial_j P(p_1/q,p_2/q)| \leq (64 B \max\{\left|\alpha_1\right|,\left|\alpha_2\right|\})^{k} N^{2 k} q^{-k \delta (\theta - \theta_0)} \,.  
\]

We use the principle that there is no integer strictly between $0$ and $1$ to obtain
\[ 1/m \leq (64 B \max\{\left|\alpha_1\right|,\left|\alpha_2\right|\})^{k} N^{2 k} q^{-k \delta (\theta - \theta_0)} \]
where $\partial_j P(p_1/q,p_2/q)=s/m$ with $s \in \BZ,m \in \BN$ and $s$ and $m$ coprime. 

Finally lemma \ref{heightdiff} gives
\[ q^{-k} \leq (64 B \max\{\left|\alpha_1\right|,\left|\alpha_2\right|\})^{k} N^{2 k} q^{-k \delta (\theta - \theta_0)} \]

and after taking $k$-th roots and simplifying we obtain
\[ q^{\delta (\theta - \theta_0)-1} \leq 64 B  \max\{\left|\alpha_1\right|,\left|\alpha_2\right|\} N^{2} \,. \]
The statement of the theorem now holds for \[C(\alpha_1,\alpha_2,\delta,N) :=\left(64 B  \max\{\left|\alpha_1\right|,\left|\alpha_2\right|\} N^{2} \right)^{\frac{1}{\delta (\theta - \theta_0)-1}} \,.\]
\end{proof}


\section{Finding a suitable global section}

For this section we will fix an embedding $\BA^2 \hookrightarrow \BP^2$ and a line $L \subset \BP^2$ such that the global sections of $\mcO_{\BP^2}(k L)$ restricted to $\BA^2$ are the polynomials of degree less or equal $k$, and view $(\alpha_1,\alpha_2)$, all of its conjugates and $(p_1/q,p_2/q)$ as elements of $\BP^2$ via this embedding. In this chapter we will always indicate which base field we are working over. 

We now state Faltings's version of Siegel's lemma. 
\begin{lemma}[{{\cite[Proposition 2.18]{F}}}]
\label{FaltingsSiegellemma}
Let $V,W$ be two finite dimensional normed $\BR$-vector spaces and let $M \subset V$ and $N \subset W$ be $\BZ$-lattices of maximal rank. Let further $\phi : V \rightarrow W$ be a linear map such that $\phi(M) \subset N$. Let $b:=\dim(V)$ and $a:=\dim(\Ker(\phi))$ and assume that there exists a constant $C \geq 2$ such that 
\begin{enumerate}[label=\arabic*.]
    \item $M$ is generated by elements of norm at most $C$,
    \item the norm of $\phi$ is bounded by $C$,
    \item all non-trivial elements of $M$ and $N$ have norm at least $1/C$.
\end{enumerate}
For $1 \leq i \leq b$ set 
\[\lambda_i := \inf \{\lambda > 0 \mid \exists \text{ $i$ linearly independent vectors of norm $\leq \lambda$ in $\Ker(\phi)\cap M$} \} \,. \]
Then it holds that \[ \lambda_{i+1} \leq (C^{3b} b!)^{1/(a-i)}\,.\]
\end{lemma}

We will need the following number theoretical lemma. 

\begin{lemma}[{{\cite[Lemma D.3.4]{HS}}}]
\label{powers}
Let $\alpha \in \overline{\BQ}$ be an algebraic integer of degree $d_{\alpha}:=[\BQ(\alpha):\BQ]$ over $\BQ$ and let $m_{\alpha} \in \BQ[X]$ be the minimal polynomial of $\alpha$ over $\BQ$. Then we have $\alpha^l = a_1^{(l)} \alpha^{d_{\alpha}-1}+\dots+a_{d_{\alpha}}^{(l)}$ with $a_i^{(l)} \in \BZ$ satisfying $\left| a_i^{(l)} \right| \leq (|m_{\alpha}|+1)^l$. 
\end{lemma}

The following statements now clarify how we intend to use Faltings's version of Siegel's lemma. 

\begin{lemma}
\label{ClarifySiegel}
Let $k \geq 6$ be a positive integer, let $B_k:=H^0(\mcO_{\BP^2_\BQ}(k L))$, which we will identify with the polynomials of degree less or equal $k$ in $\BQ[X_1,X_2]$, and let $A_k$ be the subspace of sections whose index at $(\alpha_1,\alpha_2)$ with respect to the weights $(k,k)$ is at least $\theta$. Choose an algebraic integer $\alpha$ which is a primitive element for $\BQ(\alpha_1,\alpha_2)$ and assume that $\alpha_1$ and $\alpha_2$ can be expressed as
\[
\label{expressing}
\alpha_i = c_1^i \alpha^{d-1} + \dots + c_{d-1}^i \alpha + c_d^i \text{ for $i=1,2$}
\]
where $c_h^i \in \BZ$ and let $M$ be defined as $\max\{|c_h^i| \mid h = 1,\dots,d \text{ and } i=1,2\}$ (we can always satisfy this assumption by considering $N \alpha_i$ instead of $\alpha_i$ for a suitable $N \in \BN$). Then there exists a linear map $ \phi_k : B_k \otimes \BR \rightarrow \BQ(\alpha_1,\alpha_2)^{l_k} \otimes \BR$ where \[l_k= \#\{j \in \BN_0^2 \mid \frac{j_1+j_2}{k} < \theta\}\] such that
\begin{enumerate}[label=\arabic*.]
\item $\Ker(\phi_k)=A_k \otimes \BR$,
\item $\phi_k$, the lattice inside $B_k \otimes \BR$ generated by monomials and the lattice in $\BQ(\alpha_1,\alpha_2)^{l_k} \otimes \BR$ generated by $\alpha^i$ for $i=0,\dots,d-1$ in every component satisfy the conditions in Lemma \ref{FaltingsSiegellemma} with $C=B^{k}$ where $B>0$ is the following constant
\[B:=  8 d M (|m_{\alpha}|+1)^d \,.\]
\end{enumerate}
\end{lemma}

\begin{proof}
Define the linear map
\begin{align*}
\phi_k : H^0(\mcO_{\BP^2}(k)) \otimes \BR &\rightarrow \BQ(\alpha_1,\alpha_2)^{l_k} \otimes \BR\\
P \otimes 1 &\mapsto (\partial_j P) (\alpha_1,\alpha_2) \otimes 1
\end{align*}
where $j$ ranges over all pairs of non-negative integers satisfying $(j_1+j_2)/k < \theta$. Consider the basis of $V:=B_k \otimes \BR$ which consists of monomials, the basis of $W:=\BQ(\alpha_1,\alpha_2)^{l_k} \otimes \BR$ consisting of $\alpha^i$ for $i=0,\dots,d-1$ in every component and the lattices $M$ and $N$ generated by these bases.\\
By Lemma \ref{naivediff} and the assumptions on $\alpha$ we have that $\phi_k(M) \subset N$.
We now identify $V \cong \BR^{\dim_\BQ B_k}$ and $W \cong \BR^{d l_k}$ using the above bases and equip these $\BR$-vector spaces with the maximum norm $\lvert \cdot \rvert_{\infty}$. It is then clear that $M$ is generated by elements of norm $1$ and all non-trivial elements of $M$ and $N$ have norm greater or equal $1$. Therefore we only need to give a bound on the norm of $\phi_k$.\\
To achieve this, we consider a polynomial $P \in V$, note that $P$ is a sum of at most $1/2 \, (k+1) (k+2)  \leq 2^k$ terms and use Lemma \ref{naivediff} to obtain that the coefficients of $P$ are bounded by $|\partial_j P| \leq 4^k|P|$. Then by using Lemma \ref{expressing} to expand $\alpha_1^u \alpha_2^v$ into a $\BZ$-linear combination of powers of $\alpha$ we obtain a sum of $d^{u+v} \leq d^{k}$ terms $R \alpha^l$ with $l\leq (u+v)d \leq k d$ and $R \leq M^{u+v} \leq M^{k}$. By Lemma \ref{powers} we have that $\alpha^l$ is then a $\BZ$-linear combination of $1,\alpha,\dots,\alpha^{d-1}$ with coefficients bounded by $(|m_{\alpha}|+1)^{k d}$. Therefore it holds that
\[ \lvert \phi_k(P) \rvert_{\infty} \leq |P| (8 d M (|m_{\alpha}|+1)^d)^{k}  \]
and this implies that the statement of the lemma holds with 
\[ B:= (8 d M (|m_{\alpha}|+1)^d \,.\]  
\end{proof}

\begin{lemma}
\label{asymptotic}
Keeping the notation and assumptions of the previous lemma and letting $b_k:= \dim_\BQ B_k$, $a_k:= \dim_\BQ A_k$, $i_k:= \dim_\BQ U_k$ where $U_k$ is the linear subspace of $A_k$ of sections $s \in A_k$ with $\ind_{(p_1/q,p_2/q;k,k)} s \geq \theta_0$ we have that 
\begin{align*}
\lim_{k \rightarrow \infty} \frac{b_k}{k^2/2} &= \vol_{\BP^2_\BC}(L) = 1 \, \\
\lim_{k \rightarrow \infty} \frac{a_k}{k^2/2} &= \vol_{X_\BC}(L - \theta \, (E_1 + \dots + E_d))\\
\lim_{k \rightarrow \infty} \frac{i_k}{k^2/2} &= \vol_{X_\BC}(L - \theta \, (E_1 + \dots + E_d) - \theta_0 E_{d+1})
\end{align*}
where $X_\BC$ is the blowup of $\BP^2_\BC$ in $(\alpha_1,\alpha_2)$ and all of its conjugates with corresponding exceptional divisors $E_1,\dots,E_d$ and in $(p_1/q,p_2/q)$ with corresponding exceptional divisor $E_{d+1}$. 
\end{lemma}

\begin{proof}
The first statement follows from the definition of the volume and the fact that the dimension of cohomology does not change under base change with a field \cite[Proposition 5.2.27]{Liu}.

Regarding the second and third statement note that $\theta$ and $\theta_0$ are real numbers and that the volume function for real divisors is defined by extending the volume function on $\BQ$-divisors \cite[Corollary 2.2.45]{PAG}. However by \cite[Theorem 3.5]{FKL} it holds that for a $\BR$-Cartier $\BR$-divisor $D$ on a projective variety $V$ we have \[ \vol_V(D) = \lim_{k \rightarrow \infty}  \frac{h^0(\lfloor k D\rfloor)}{k^{\dim(V)}/\dim(V)!}\,. \]  

For the second statement let us consider the $\overline{\BQ}$-point $(\alpha_1,\alpha_2)$ and its image in $\BA^2_\BQ$. This is a closed point in $\BA^2_\BQ$ and it corresponds to a maximal ideal $I_{(\alpha_1,\alpha_2)} \subset \BQ[X_1,X_2]$. We will denote the corresponding ideal sheaf in $\BP^2_\BQ$ by $\mcI_{(\alpha_1,\alpha_2)}$. Now the set of all polynomials in $\BQ[X_1,X_2]$ such that $(\partial_j P) (\alpha_1,\alpha_2) = 0$ for all pairs of non-negative integers satisfying $(j_1+j_2)/k < \theta$ is by definition the so called $\lceil k \theta\rceil$-th differential power of $I_{(\alpha_1,\alpha_2)}$ and by \cite[Theorem 2.12, Proposition 2.14]{Dao} this is just $I_{(\alpha_1,\alpha_2)}^{\lceil k\theta\rceil}$ as $I_{(\alpha_1,\alpha_2)}$ is maximal (note that this is not true for general ideals \cite{Szemberg} but there are results on the difference between symbolic and ordinary powers of ideals \cite{ELS,Huneke}).\\
Now consider the sheaf $\mcO_{\BP^2_\BQ}(k L) \mcI_{(\alpha_1,\alpha_2)}^{\lceil k\theta\rceil}$ whose global sections restricted to $\BA^2$ are by the above consideration exactly the elements of $\BQ[X_1,X_2]$ of degree less or equal $k$ being mapped to zero by $\phi_k$. Note that $\mcO_{\BP^2_\BQ}(k L) \mcI_{(\alpha_1,\alpha_2)}^{\lceil k\theta\rceil}$ is isomorphic to $\mcO_{\BP^2_\BQ}(k L) \otimes \mcI_{(\alpha_1,\alpha_2)}^{\lceil k\theta\rceil}$ as $\mcO_{\BP^2_\BQ}(k L)$ is locally free and therefore flat. The base change morphism $\rho: \BP^2_\BC \rightarrow \BP^2_\BQ$ is flat by \cite[Proposition 4.3.3 (d)]{Liu} and thus $\rho^* \mcI_{(\alpha_1,\alpha_2)}$ is the ideal sheaf of the preimage of $(\alpha_1,\alpha_2)$ under $\rho$, which is simply the ideal sheaf $\mcJ$ of the union of all conjugates of $(\alpha_1,\alpha_2)$. As above we can invoke \cite[Proposition 5.2.27]{Liu} to obtain $H^0(\mcO_{\BP^2_\BQ}(k L) \otimes \mcI_{(\alpha_1,\alpha_2)}^{\lceil k\theta\rceil}) = H^0(\mcO_{\BP^2_\BC}(k L) \otimes \mcJ^{\lceil k\theta\rceil})$ and by \cite[Lemma 4.3.16]{PAG} this is isomorphic to $H^0(\mcO_X(\lfloor k L - k \theta \, (E_1 + \dots + E_d) \rfloor)$, which finishes the proof of the second statement.
 
For the third statement let $I_{(p_1/q,p_2/q)} \subset \BQ[X_1,X_2]$ be the ideal of $(p_1/q,p_2/q)$. We will denote the corresponding ideal sheaf in $\BP^2_\BQ$ by $\mcI_{(p_1/q,p_2/q)}$. We may now argue as above to obtain that $H^0(\mcO_X(\lfloor k L - k \theta \, (E_1 + \dots + E_d) - \theta_0 E_{d+1} \rfloor)$ is isomorphic to the global sections of the sheaf $\mcO_{\BP^2_\BQ}(k L) \mcI_{(\alpha_1,\alpha_2)}^{\lceil k\theta\rceil} \mcI_{(p_1/q,p_2/q)}^{\lceil k\theta_0\rceil}$.\\
As both $I_{(p_1/q,p_2/q)}$ and $I_{(\alpha_1,\alpha_2)}$ are maximal ideals in $\BQ[X_1,X_2]$ we have further that $I_{(p_1/q,p_2/q)}+I_{(\alpha_1,\alpha_2)}=\BQ[X_1,X_2]$ and by taking both sides to the $u+v$-th power we also obtain that $I_{(p_1/q,p_2/q)}^u+I_{(\alpha_1,\alpha_2)}^v=\BQ[X_1,X_2]$ for all $u,v \in \BN$. By the usual argument this implies that \[I_{(p_1/q,p_2/q)}^u \cap I_{(\alpha_1,\alpha_2)}^v = I_{(p_1/q,p_2/q)}^u I_{(\alpha_1,\alpha_2)}^v \,.\] Therefore we conclude that $U_k$ is isomorphic to the global sections of \[\mcO_{\BP^2_\BQ}(k L) \mcI_{(\alpha_1,\alpha_2)}^{\lceil k\theta\rceil} \mcI_{(p_1/q,p_2/q)}^{\lceil k\theta_0\rceil}\,,\] which in turn are isomorphic to $H^0(\mcO_X(\lfloor k L - k \theta \, (E_1 + \dots + E_d) - k \theta_0 E_{d+1} \rfloor)$. This concludes the proof.      
\end{proof}


\section{Proof of the main theorem}

Let us first fix some notation for this section. Let $\alpha_0$ be defined as $\alpha_1 + M_0 \alpha_2$ where $M_0$ is the smallest natural number that $\alpha_1 + M_0 \alpha_2$ is a primitive element of $\BQ(\alpha_1,\alpha_2)$ (such a $M_0$ always exists by the proof of the primitive element theorem \cite[Theorem V.4.6]{LangAlgebra}). Let $\alpha$ be defined as $M_1 \alpha_0$ where $M_1$ is the smallest natural number such that $M_1 \alpha_0$ is an algebraic integer. Now let $N$ be the smallest natural number such that $N \alpha_1$ and $N \alpha_2$ can be expressed as\[
N \alpha_i = c_1^i \alpha^{d-1} + \dots + c_{d-1}^i \alpha + c_d^i \text{ for $i=1,2$}
\]
where $c_h^i \in \BZ$ and let $M$ be defined as $\max\{|c_h^i| \mid h = 1,\dots,d \text{ and } i=1,2\}$. Let $Q$ be defined as the denominator of \[\theta:=\frac{1/\delta+\min\{1/m,m/d\}}{2}\] and let \[ \theta_0:= \min\left\{\frac{\min\{1/m,m/d\}-1/\delta}{4},\frac{1}{Q l(\theta)}\right\} \] where \[l(\theta):=\frac{1}{2} \left\lceil \frac{((-3+d \theta)Q+4 Q^2 (1-d \theta^2)+1)^2}{Q^2(1-d \theta^2)} + 3\right\rceil .\]

We are now ready to conclude the proof of the main theorem.

\begin{proof}[Proof of the main theorem]
Let us consider the asymptotics obtained in the Lemma \ref{asymptotic} above and use Faltings's version of Siegel's lemma. In order to use Lemma \ref{ClarifySiegel} we replace $\alpha_i$ by $N \alpha_i$. After this replacement we have that \[\left|N \alpha_i - \frac{N p_i}{q} \right| \leq N q^{-\delta} \text{ for $i=1,2$}\,.\]

We have by Lemma \ref{ClarifySiegel} that $B^{k}$ satisfies the assumptions on $C$ in Lemma \ref{FaltingsSiegellemma} and therefore
\begin{equation}
\label{eq:bound}
 \lambda_{i_k+1} \leq (( 8 d M (|m_{\alpha}|+1)^d)^{3 k b_k} b_k!)^{1/(a_k-i_k)} \leq ((8 d M (|m_{\alpha}|+1)^d)^{3 k} b_k)^{b_k/(a_k-i_k)} \,. 
\end{equation}

By the choice of $\theta$ and $\theta_0$, Lemma \ref{asymptotic} and Theorem \ref{voldiff} the exponent on the right hand side of \ref{eq:bound} satisfies
\begin{align*} &\lim_{k \rightarrow \infty} \frac{b_k}{(a_k-i_k)}\\ ={}&  \frac{\vol_{\BP^2}(L)}{\vol_X(L - \theta (E_1 + \dots + E_d)) - \vol_X(L - \theta (E_1 + \dots + E_d) - \theta_0 E_{d+1})}\\
	={} &1/\theta_0^2 \,.
\end{align*} 

Now Lemma \ref{FaltingsSiegellemma} shows the existence of $i_k+1$ linearly independent elements of $A_k$ such that their norm is bounded by \[((8 d M (|m_{\alpha}|+1)^d)^{3 k} b_k)^{b_k/(a_k-i_k)} \leq (2^{10} (d M (|m_{\alpha}|+1)^d)^{3} )^{k b_k/(a_k-i_k)}\] where we have used that $b_k=1/2 \, (k+1) (k+2)  \leq 2^k$. In particular, for $k\gg 0$ at least one of those elements $P$ is not an element of $U_k$. Noting that \[\delta (\theta - \theta_0) \geq  1/4 \, \delta \min\{1/m,m/d\} +3/4 > 1\,,\] we conclude that $P$ satisfies all of the conditions for Lemma \ref{lower} and we obtain that 
\[q \leq \left(64 (2^{10} (d M (|m_{\alpha}|+1)^d)^{3} )^{b_k/(a_k-i_k)}  \max\{\left|\alpha_1\right|,\left|\alpha_2\right|\} N^{3} \right)^{\frac{1}{\delta (\theta - \theta_0)-1}}\,.
\]
Finally we take the limit for $k \rightarrow \infty$ and obtain
\[
q \leq \left(64 (2^{10} (d M (|m_{\alpha}|+1)^d)^{3} )^{1/\theta_0^2}  \max\{\left|\alpha_1\right|,\left|\alpha_2\right|\} N^{3} \right)^{\frac{1}{\delta (\theta - \theta_0)-1}}\,,
\]
which finishes the proof.
\end{proof}

The proof of Theorem \ref{thm:main} immediately yields the following corollary concerning a possible value of $C_0(\alpha_1,\alpha_2,\delta,m)$.

\begin{cor}
\label{cor:C0}
In Theorem \ref{thm:main} we can take
\[ C_0(\alpha_1,\alpha_2,\delta) :=  \left(64 (2^{10} (d M (|m_{\alpha}|+1)^d)^{3} )^{1/\theta_0^2}  \max\{\left|\alpha_1\right|,\left|\alpha_2\right|\} N^{3} \right)^{\frac{1}{\delta (\theta - \theta_0)-1}}.\]
\end{cor}

\bibliographystyle{amsalpha}
\bibliography{Bibliography}

\newcommand{\etalchar}[1]{$^{#1}$}
\providecommand{\bysame}{\leavevmode\hbox to3em{\hrulefill}\thinspace}
\providecommand{\MR}{\relax\ifhmode\unskip\space\fi MR }
\providecommand{\MRhref}[2]{%
  \href{http://www.ams.org/mathscinet-getitem?mr=#1}{#2}
}
\providecommand{\href}[2]{#2}
\begin{thebibliography}{BDRH{\etalchar{+}}09}

\bibitem[Bak67]{BakerSimul}
A.~Baker, \emph{Simultaneous rational approximations to certain algebraic
  numbers}, Proc. Cambridge Philos. Soc. \textbf{63} (1967), 693--702.
  \MR{213303}

\bibitem[BDRH{\etalchar{+}}09]{primer}
Thomas Bauer, Sandra Di~Rocco, Brian Harbourne, Micha\l Kapustka, Andreas
  Knutsen, Wioletta Syzdek, and Tomasz Szemberg, \emph{A primer on {S}eshadri
  constants}, Interactions of classical and numerical algebraic geometry,
  Contemp. Math., vol. 496, Amer. Math. Soc., Providence, RI, 2009, pp.~33--70.
  \MR{2555949}

\bibitem[Ben95]{Bennett}
Michael~A. Bennett, \emph{Simultaneous approximation to pairs of algebraic
  numbers}, Number theory ({H}alifax, {NS}, 1994), CMS Conf. Proc., vol.~15,
  Amer. Math. Soc., Providence, RI, 1995, pp.~55--65. \MR{1353921}

\bibitem[BG96]{BuGy}
Yann Bugeaud and K\'{a}lm\'{a}n Gy\H{o}ry, \emph{Bounds for the solutions of
  {T}hue-{M}ahler equations and norm form equations}, Acta Arith. \textbf{74}
  (1996), no.~3, 273--292. \MR{1373714}

\bibitem[Bom93]{BoGm}
Enrico Bombieri, \emph{Effective {D}iophantine approximation on
  {${\mathbf{G}}_m$}}, Ann. Scuola Norm. Sup. Pisa Cl. Sci. (4) \textbf{20}
  (1993), no.~1, 61--89. \MR{1215999}

\bibitem[B\u01]{Badescu}
Lucian B\u{a}descu, \emph{Algebraic surfaces}, Universitext, Springer-Verlag,
  New York, 2001, Translated from the 1981 Romanian original by Vladimir
  Ma\c{s}ek and revised by the author. \MR{1805816}

\bibitem[Bug98]{BuBorne}
Yann Bugeaud, \emph{Bornes effectives pour les solutions des \'{e}quations en
  {$S$}-unit\'{e}s et des \'{e}quations de {T}hue-{M}ahler}, J. Number Theory
  \textbf{71} (1998), no.~2, 227--244. \MR{1633809}

\bibitem[CZ04]{CorvajaZannier}
Pietro Corvaja and Umberto Zannier, \emph{On a general {T}hue's equation},
  Amer. J. Math. \textbf{126} (2004), no.~5, 1033--1055. \MR{2089081}

\bibitem[DDSG{\etalchar{+}}18]{Dao}
Hailong Dao, Alessandro De~Stefani, Elo\'{\i}sa Grifo, Craig Huneke, and Luis
  N\'{u}\~{n}ez Betancourt, \emph{Symbolic powers of ideals}, Singularities and
  foliations. geometry, topology and applications, Springer Proc. Math. Stat.,
  vol. 222, Springer, Cham, 2018, pp.~387--432. \MR{3779569}

\bibitem[Dem92]{Demailly}
Jean-Pierre Demailly, \emph{Singular {H}ermitian metrics on positive line
  bundles}, Complex algebraic varieties ({B}ayreuth, 1990), Lecture Notes in
  Math., vol. 1507, Springer, Berlin, 1992, pp.~87--104. \MR{1178721}

\bibitem[Dys47]{Dyson}
F.~J. Dyson, \emph{The approximation to algebraic numbers by rationals}, Acta
  Math. \textbf{79} (1947), 225--240.

\bibitem[EF02]{EF}
Jan-Hendrik Evertse and Roberto~G. Ferretti, \emph{Diophantine inequalities on
  projective varieties}, Int. Math. Res. Not. (2002), no.~25, 1295--1330.
  \MR{1903776}

\bibitem[EF08]{EF2}
\bysame, \emph{A generalization of the {S}ubspace {T}heorem with polynomials of
  higher degree}, Diophantine approximation, Dev. Math., vol.~16,
  SpringerWienNewYork, Vienna, 2008, pp.~175--198. \MR{2487693}

\bibitem[ELS01]{ELS}
Lawrence Ein, Robert Lazarsfeld, and Karen~E. Smith, \emph{Uniform bounds and
  symbolic powers on smooth varieties}, Invent. Math. \textbf{144} (2001),
  no.~2, 241--252. \MR{1826369}

\bibitem[EV84]{EV}
H{\'e}l{\`e}ne Esnault and Eckart Viehweg, \emph{Dyson's lemma for polynomials
  in several variables (and the theorem of {Roth})}, Inventiones mathematicae
  \textbf{78} (1984), 445--490.

\bibitem[{Fal}91]{F}
Gerd {Faltings}, \emph{{Diophantine approximation on abelian varieties.}},
  {Ann. Math. (2)} \textbf{133} (1991), no.~3, 549--576.

\bibitem[FdB96]{Fernandez}
Guillermo Fern\'{a}ndez~del Busto, \emph{A {M}atsusaka-type theorem on
  surfaces}, J. Algebraic Geom. \textbf{5} (1996), no.~3, 513--520.
  \MR{1382734}

\bibitem[Fel71]{Feldman}
N.~I. Fel'dman, \emph{An effective power sharpening of a theorem of
  {L}iouville}, Izv. Akad. Nauk SSSR Ser. Mat. \textbf{35} (1971), 973--990.
  \MR{0289418}

\bibitem[Fer00]{FerrettiMumford}
Roberto~G. Ferretti, \emph{Mumford's degree of contact and {D}iophantine
  approximations}, Compositio Math. \textbf{121} (2000), no.~3, 247--262.
  \MR{1761626}

\bibitem[FKL16]{FKL}
Mihai Fulger, J\'{a}nos Koll\'{a}r, and Brian Lehmann, \emph{Volume and
  {H}ilbert function of {$\Bbb R$}-divisors}, Michigan Math. J. \textbf{65}
  (2016), no.~2, 371--387. \MR{3510912}

\bibitem[FW94]{FW}
Gerd {Faltings} and Gisbert {W\"ustholz}, \emph{{Diophantine approximations on
  projective spaces.}}, {Invent. Math.} \textbf{116} (1994), no.~1-3, 109--138
  (English).

\bibitem[HH02]{Huneke}
Melvin Hochster and Craig Huneke, \emph{Comparison of symbolic and ordinary
  powers of ideals}, Invent. Math. \textbf{147} (2002), no.~2, 349--369.
  \MR{1881923}

\bibitem[HL17]{HeierLevin}
Gordon Heier and Aaron Levin, \emph{A generalized {S}chmidt subspace theorem
  for closed subschemes}, arXiv preprint arXiv:1712.02456 (2017).

\bibitem[HS00]{HS}
Marc Hindry and Joseph~H. Silverman, \emph{Diophantine geometry}, Graduate
  Texts in Mathematics, vol. 201, Springer-Verlag, New York, 2000.

\bibitem[KLM12]{KLM12}
Alex K\"uronya, Victor Lozovanu, and Catriona Maclean, \emph{Convex bodies
  appearing as {O}kounkov bodies of divisors}, Adv. Math. \textbf{229} (2012),
  no.~5, 2622--2639. \MR{2889138}

\bibitem[Lan02]{LangAlgebra}
Serge Lang, \emph{Algebra}, third ed., Graduate Texts in Mathematics, vol. 211,
  Springer-Verlag, New York, 2002. \MR{1878556}

\bibitem[Laz04]{PAG}
Robert Lazarsfeld, \emph{Positivity in algebraic geometry. {I}}, Ergebnisse der
  Mathematik und ihrer Grenzgebiete. 3. Folge. A Series of Modern Surveys in
  Mathematics [Results in Mathematics and Related Areas. 3rd Series. A Series
  of Modern Surveys in Mathematics], vol.~48, Springer-Verlag, Berlin, 2004,
  Classical setting: line bundles and linear series. \MR{2095471}

\bibitem[Lio51]{Liouville}
J.~Liouville, \emph{Sur des classes très-étendues de quantités dont la
  valeur n'est ni algébrique, ni même réductible à des irrationnelles
  algébriques.}, Journal de Mathématiques Pures et Appliquées (1851),
  133--142 (fre).

\bibitem[Liu02]{Liu}
Qing Liu, \emph{Algebraic geometry and arithmetic curves}, Oxford Graduate
  Texts in Mathematics, vol.~6, Oxford University Press, Oxford, 2002,
  Translated from the French by Reinie Ern\'{e}, Oxford Science Publications.
  \MR{1917232}

\bibitem[Mat70]{Matsusaka2}
T.~Matsusaka, \emph{On canonically polarized varieties. {II}}, Amer. J. Math.
  \textbf{92} (1970), 283--292. \MR{263816}

\bibitem[MM64]{MM}
T.~Matsusaka and D.~Mumford, \emph{Two fundamental theorems on deformations of
  polarized varieties}, Amer. J. Math. \textbf{86} (1964), 668--684.
  \MR{171778}

\bibitem[MR15]{mcroth}
David McKinnon and Mike Roth, \emph{Seshadri constants, diophantine
  approximation, and roth’s theorem for arbitrary varieties}, Inventiones
  mathematicae \textbf{200} (2015), no.~2, 513--583.

\bibitem[Osg70]{Osgood}
Charles~F. Osgood, \emph{The simultaneous diophantine approximation of certain
  {$k{\rm th}$} roots}, Proc. Cambridge Philos. Soc. \textbf{67} (1970),
  75--86. \MR{249367}

\bibitem[Ric]{Rickert}
John~H. Rickert, \emph{Simultaneous rational approximations and related
  diophantine equations}, Math. Proc. Cambridge Philos. Soc. \textbf{113},
  no.~3, 461--472.

\bibitem[Rot55]{KFRoth}
K.~F. Roth, \emph{Rational approximations to algebraic numbers}, Mathematika
  \textbf{2} (1955), no.~1, 1–20.

\bibitem[RV16]{RuVojta}
Min Ru and Paul Vojta, \emph{A birational nevanlinna constant and its
  consequences}, arXiv preprint arXiv:1608.05382 (2016).

\bibitem[RW17]{RuWang}
Min Ru and Julie Tzu-Yueh Wang, \emph{A subspace theorem for subvarieties},
  Algebra Number Theory \textbf{11} (2017), no.~10, 2323--2337. \MR{3744358}

\bibitem[Sch70]{Schmidt}
Wolfgang~M. Schmidt, \emph{Simultaneous approximation to algebraic numbers by
  rationals}, Acta Math. \textbf{125} (1970), 189--201. \MR{0268129}

\bibitem[Ses72]{Seshadri}
C.~S. Seshadri, \emph{Quotient spaces modulo reductive algebraic groups}, Ann.
  of Math. (2) \textbf{95} (1972), 511--556; errata, ibid. (2) 96 (1972), 599.
  \MR{0309940}

\bibitem[Siu93]{Siu}
Yum~Tong Siu, \emph{An effective {M}atsusaka big theorem}, Ann. Inst. Fourier
  (Grenoble) \textbf{43} (1993), no.~5, 1387--1405. \MR{1275204}

\bibitem[Siu02]{Siu2}
Yum-Tong Siu, \emph{A new bound for the effective {M}atsusaka big theorem},
  vol.~28, 2002, Special issue for S. S. Chern, pp.~389--409. \MR{1898197}

\bibitem[SS]{Szemberg}
T.~Szemberg and J.~Szpond, \emph{On the containment problem}, Rend. Circ. Mat.
  Palermo (2) \textbf{66}, no.~2, 233--245.

\bibitem[Voj91]{VojtaSiegel}
Paul Vojta, \emph{Siegel's theorem in the compact case}, Annals of Mathematics
  \textbf{133} (1991), no.~3, 509--548.

\end{thebibliography}

\end{document}